\documentclass[letterpaper,11pt]{article}

\usepackage{layout}
\usepackage[pdftex]{graphics}  
\usepackage{graphicx}

\usepackage{hyperref}
\usepackage{amsmath, amsthm, amssymb}
\usepackage{booktabs}
\usepackage{placeins}

\newtheorem{thm}{Theorem}[section]
\newtheorem{cor}[thm]{Corollary}

\newtheorem{prop}[thm]{Proposition}
\theoremstyle{definition}

\theoremstyle{remark}

\numberwithin{equation}{section}

\title{\bf Sparse pseudoinverses via LP and SDP \\ relaxations of Moore-Penrose\footnote{Partial support from  NSF grant CMMI--1160915 and ONR grant N00014-14-1-0315.}}

\author{%
Victor K. Fuentes\\%
University of Michigan\\%
\textit{vicfuen@umich.edu}\\[3mm]
Marcia Fampa\\%
Universidade Federal do Rio de Janeiro\\%
\textit{fampa@cos.ufrj.br}\\[3mm]
Jon Lee\\%
University of Michigan\\%
\textit{jonxlee@umich.edu}
}

\date{\today}

\begin{document}

\maketitle

\thispagestyle{empty}

\begin{abstract}

Pseudoinverses are ubiquitous tools for handling over- and under-determined
systems of equations. For computational efficiency,
sparse pseudoinverses are
desirable. Recently, sparse left and right pseudoinverses were introduced,
using 1-norm minimization and linear programming. We introduce several new
sparse pseudoinverses by developing
linear and semi-definite programming relaxations of the well-known
Moore-Penrose properties.

\vspace{2em}
\noindent\textit{Keywords:}
sparse;
pseudoinverse;
semi-definite programming;
linear programming.

\end{abstract}

\section*{Introduction}\label{sec:intro}

Pseudoinverses are a central tool in matrix algebra and its applications.
Sparse optimization is concerned with finding sparse solutions of
optimization problems, often for computational efficiency in the use of the output of the optimization.
There is usually a tradeoff between an ideal dense solution and a less-ideal sparse solution,
and sparse optimization is often focused on tractable methods for striking a good balance.
Recently, sparse optimization has been used to calculate tractable sparse left and right pseudoinverses, via linear programming.
We extend this theme to derive several other tractable sparse pseudoinverses, employing linear and semi-definite programming.

In \S\ref{sec:pseudo}, we give a very brief overview of pseudoinverses, and
in  \S\ref{sec:sparse}, we describe
some prior work on sparse left and right pseudoinverses.
In  \S\ref{sec:new},
we present new sparse pseudoinverses based on tractable convex relaxations of
the Moore-Penrose properties. In  \S\ref{sec:comp}, we present preliminary computational
results.
Finally, in \S\ref{sec:conclusion}, we make brief conclusions and describe our ongoing work.


In what follows, $I_{n}$/$\textbf{0}_{n}$/$\vec{0}_{n}$/$\vec{1}_{n}$ denotes an order-$n$ identity matrix/zero matrix/zero vector/all-ones vector.

\section{Pseudoinverses}\label{sec:pseudo}

When a real matrix $A\in\mathbb{R}^{m\times n}$ is not square or not invertible,
we consider pseudoinverses  of $A$ (see \cite{rao1971} for a wealth of information on this topic).
For example, there is the well-known \emph{Drazin inverse} for square and even non-square matrices
(see \cite{CG1980}) and the \emph{generalized Bott-Duffin inverse} (see \cite{Yonglin1990}).

The most well-known pseudoinverse of all is the \emph{M-P (Moore-Penrose) pseudoinverse},
independently discovered by  E.H. Moore and R. Penrose.
If $A=U\Sigma V'$ is the real singular value decomposition of $A$ (see \cite{GVL1996}, for example),
then the
M-P pseudoinverse of $A$ can be defined as $A^+:=V\Sigma^+ U'$, where $\Sigma^+$
has the shape of the transpose of the diagonal matrix $\Sigma$, and is derived from $\Sigma$
by taking reciprocals of the non-zero (diagonal) elements of $\Sigma$ (i.e., the non-zero
singular values of $A$).
The M-P pseudoinverse, a central object in matrix theory, has many
concrete uses. For example, we can use it to solve least-squares problems,
 and we can use it, together
with a norm, to define condition numbers of matrices.
The M-P pseudoinverse
is calculated, via its connection with the real singular value decomposition,
by the \verb;Matlab; function \verb;pinv;.

\section{Sparse left and right pseudoinverses}\label{sec:sparse}

It is well known that in the context of seeking a sparse solution in a convex set,
a surrogate for minimizing the sparsity is to minimize the 1-norm. In fact,
if the components of the solution have absolute value no more than unity,
a minimum 1-norm solution has 1-norm no greater than the number of nonzeros
in the sparsest solution.
With this in mind, \cite{pseudoinverse} defines sparse left and right pseudoinverses
in a natural and tractable manner. Below, $\|\cdot\|_1$ denotes entry-wise 1-norm.

For an ``overdetermined case'', \cite{pseudoinverse} defines a \emph{sparse left pseudoinverse} via
the convex
formulation
\begin{equation} \tag{$\mathcal{O}$} \label{overdetermined}
\min \left\{
\|H\|_{1} ~:~ HA = I_{n}
\right\}.
\end{equation}
For an ``underdetermined case'', \cite{pseudoinverse} defines  a \emph{sparse right pseudoinverse} via
the convex
formulation
\begin{equation} \tag{$\mathcal{U}$} \label{underdetermined}
\min \left\{
\|H\|_{1} ~:~ AH = I_{m}
\right\}.
\end{equation}

These definitions emphasize sparsity, while in some sense
putting a rather mild emphasis on the aspect of being a pseudoinverse.
We do note that if the columns of $A$ are linearly independent, then
the M-P pseudoinverse is precisely $(A'A)^{-1}A'$, which is
a left inverse of $A$. Therefore, if $A$ has full column rank, then the M-P pseudoinverse
is a feasible $H$ for ($\mathcal{O}$). Conversely, if $A$ does not have full column rank,
then ($\mathcal{O}$) has no feasible solution, and so there is no sparse left inverse in such a case.
On the other hand,  if the rows of $A$ are linearly independent, then
the M-P pseudoinverse is precisely $A'(AA')^{-1}$,
which is a right inverse of $A$.
 Therefore, if $A$ has full row rank, then the M-P pseudoinverse
is a feasible $H$ for ($\mathcal{U}$). Conversely, if $A$ does not have full row rank,
then ($\mathcal{U}$) has no feasible solution, and so there is no sparse right inverse in such a case.

These sparse pseudoinverses are easy to calculate, by linear programming:
\[
 \tag{$LP_{\mathcal{O}}$} \label{firstLP}
\min \left\{\textstyle\sum_{ij\in m\times n} ~t_{ij} ~:~ t_{ij} \geq h_{ij},~ t_{ij} \geq -h_{ij},~ \forall ij\in m\times n;~ HA = I_n \right\}
\]
for the sparse left pseudoinverse, and
\[
\tag{$LP_{\mathcal{U}}$} \label{secondLP}
\min \left\{\textstyle\sum_{ij\in m\times n} ~t_{ij} ~:~ t_{ij} \geq h_{ij},~ t_{ij} \geq -h_{ij},~ \forall ij\in m\times n;~ AH = I_m \right\}
\]
for the sparse right pseudoinverse. In fact, the (\ref{firstLP}) decomposes row-wise for $H$,
and (\ref{secondLP}) decomposes column-wise for $H$, so calculating these sparse pseudoinverses can be
made very efficient at large scale.
Also, these sparse pseudoinverses do have nice mathematical properties
(see \cite{pseudoinverse}).

\section{New sparse relaxed Moore-Penrose pseudoinverses}\label{sec:new}

We seek to define different tractable sparse pseudoinverses, based on the
 the following nice characterization of the
M-P pseudoinverse.

\begin{thm}
For $A$ $\in$ $R^{m \times n}$, the M-P pseudoinverse $A^+$ is the unique $H$ $\in$ $\mathbb{R}^{n \times m}$ satisfying:
	\begin{align}
		& AHA = A \label{property1} \tag{P1}\\
		& HAH = H \label{property2} \tag{P2}\\
		& (AH)^{\prime} = AH \label{property3} \tag{P3}\\
		& (HA)^{\prime} = HA \label{property4} \tag{P4}
	\end{align}
\end{thm}

If we consider properties (\ref{property1}) - (\ref{property4}) which characterize the M-P pseudoinverse, we can observe that properties (\ref{property1}), (\ref{property3}) and (\ref{property4}) are all linear in $H$, and the only non-linearity is property (\ref{property2}), which is quadratic. Another important point to observe is that without property (\ref{property1}), $H$ could be the all-zero matrix and satisfy properties (\ref{property2}), (\ref{property3}) and (\ref{property4}). Whenever property (\ref{property1}) holds, $H$ is called a \emph{generalized inverse}.
So, in the simplest approach, we can consider minimizing $\|H\|_{1}$ subject to property (\ref{property1}) and \emph{any subset} of the properties (\ref{property3}) and (\ref{property4}). In this manner, we get several (four) new sparse pseudoinverses which can all be calculated by linear programming.

To go further, we can also consider convex relaxations of property (\ref{property2}). To pursue that direction,
we enter the realm of semi-definite programming (see \cite{BV2004},\cite{SDP},\cite{sdphandbook}, for example).

We can see  property (\ref{property2}) as
\[
h_{i\cdot} A h_{\cdot j}=h_{ij},
\]
for all $ij\in m\times n$.
So, we have $mn$ quadratic equations to enforce, which
we can see as
\begin{align}\label{ha_inprod}
		\frac{1}{2}\left(h_{i\cdot}, h_{\cdot j}^{\prime}\right)\left[\begin{array}{cc}
										   \textbf{0}_{m} & A \\
										   A' & \textbf{0}_{n}
										   \end{array}\right]
								            \left(\begin{array}{c}
								            	   h_{i\cdot}^{\prime} \\
									   	   h_{\cdot j}
										   \end{array}\right) = h_{ij},
	\end{align}
for all $ij\in m\times n$.
We can view these quadratic equations (\ref{ha_inprod}) as
\[
 \frac{1}{2} \left\langle Q, \left(\begin{array}{c}
								            	   h_{i\cdot}^{\prime} \\
									   	   h_{\cdot j}
										   \end{array}\right)\left(h_{i\cdot}, h_{\cdot j}^{\prime}\right)
\right\rangle = h_{ij},
\]
for all $ij\in m\times n$,
where
\[
Q := \left[\begin{array}{cc}
										   \textbf{0}_{m} & A \\
										   A' & \textbf{0}_{n}
										   \end{array}\right] \in \mathbb{R}^{(m+n)\times (m+n)},
\]
and $\left\langle \cdot, \cdot \right\rangle$
denotes element-wise dot-product.

Now, we lift the variables to matrix space, defining matrix variables
\[
\mathcal{H}_{ij}:= \left(\begin{array}{c}
								            	   h_{i\cdot}^{\prime} \\
									   	   h_{\cdot j}
										   \end{array}\right)\left(h_{i\cdot}, h_{\cdot j}^{\prime}\right)\in \mathbb{R}^{(m+n)\times (m+n)},
\]
for all $ij\in m\times n$.
So, we can see (\ref{ha_inprod}) as the \emph{linear} equations
\begin{equation} \label{property2a}
\frac{1}{2}\left\langle Q, \mathcal{H}_{ij}\right\rangle = h_{ij},
\end{equation}
for all $ij\in m\times n$,
together with the \emph{non-convex} equations
\begin{equation} \label{property2b}
\mathcal{H}_{ij} - \left(\begin{array}{c}
								            	   h_{i\cdot}^{\prime} \\
									   	   h_{\cdot j}
										   \end{array}\right)\left(h_{i\cdot}, h_{\cdot j}^{\prime}\right) = \textbf{0}_{m+n},
\end{equation}
for all $ij\in m\times n$. Next, we relax the  equations (\ref{property2b}) via the \emph{convex semi-definiteness} constraints:
\begin{equation} \label{property2b_r}
\mathcal{H}_{ij} - \left(\begin{array}{c}
								            	   h_{i\cdot}^{\prime} \\
									   	   h_{\cdot j}
										   \end{array}\right)\left(h_{i\cdot}, h_{\cdot j}^{\prime}\right)  \succeq \textbf{0}_{m+n},
\end{equation}
for all $ij\in m\times n$.
So we can relax the M-P property (\ref{property2}) as
(\ref{property2a}) and (\ref{property2b_r}), for all $ij\in m\times n$.

To put (\ref{property2b_r}) into a standard form for semi-definite programming,
we create variables vectors $x_{ij}\in\mathbb{R}^{m+n}$,
and we have linear equations
\begin{equation}
x_{ij}=\left(\begin{array}{c}
								            	   h_{i\cdot}^{\prime} \\
									   	   h_{\cdot j}
										   \end{array}\right).
\end{equation}

Next, for all $ij\in m\times n$, we introduce symmetric positive semi-definite matrix variables $Z_{ij}\in \mathbb{R}^{(m+n + 1) \times (m+n+1)}$, interpreting the entries as follows:
\begin{equation}
Z_{ij} = \left[\begin{array}{cc}
				   x_{ij}^{(0)} & x_{ij}^{\prime} \\
				   x_{ij} & \mathcal{H}_{ij}
				   \end{array}\right].
\end{equation}
Then the linear equation
\begin{equation}
x_{ij}^{(0)} =1
\end{equation}
and $Z_{ij}
\succeq \textbf{0}_{(m+n+1)\times (m+n+1)}$
precisely enforce (\ref{property2b_r}).

Finally, we re-cast (\ref{property2a}) as
\begin{equation} \label{property2a_p}
\frac{1}{2}\left\langle \bar{Q}, Z_{ij}\right\rangle = h_{ij},
\end{equation}
where
\begin{equation}
\bar{Q}:=\left[\begin{array}{cc}
				      0 & \vec{0}_{m+n}^{\prime} \\
				      \vec{0}_{m+n} & Q
				      \end{array}\right] \in \mathbb{R}^{(m+n + 1) \times (m+n+1)}.
\end{equation}

In summary, we can consider minimizing $\|H\|_{1}$ subject to property (\ref{property1}) and \emph{any subset} of  (\ref{property3}), (\ref{property4}), and (\ref{property2a})+(\ref{property2b_r}) for all $ij\in m\times n$
(though we reformulate (\ref{property2a})+(\ref{property2b_r}) as above, so it is in a convenient form for
semi-definite programming solvers like \verb;CVX;).
In doing so, we get many (eight)
 new sparse pseudoinverses which  are all tractable (via linear or semi-definite programming).

Of course all of these optimization problems have feasible solutions, because
the M-P pseudoinverse $A^+$ always gives a feasible solution.
For the cases in which we have semi-definite programs,
an important issue is whether there is a strictly feasible solution --- the \emph{Slater
condition(/constraint qualification)} --- as
that is sufficient for strong duality to hold and affects the
convergence of algorithms (e.g., see \cite{BV2004}). Even if the Slater condition
does not hold, there is a facial-reduction algorithm that can
induce the Slater condition to hold on an appropriate face of the feasible region (see \cite{Pataki2013}).

\section{Preliminary computational experiments}\label{sec:comp}

We made some preliminary tests of our ideas, using CVX/Matlab (see \cite{cvx1}, \cite{cvx2}).
Before describing our experimental setup, we observe the following results.

\begin{prop}
If $A$ has full column rank $n$ and $H$ satisfies (\ref{property1}), then $H$  is a left inverse of $A$,
and $H$ satisfies (\ref{property2}) and (\ref{property4}).
If $A$ has full row rank $m$ and $H$ satisfies (\ref{property1}), then $H$ is a right inverse of $A$, and
$H$ satisfies (\ref{property2}) and (\ref{property3}),
\end{prop}

\begin{cor}
If $A$ has full column rank $n$ and $H$ satisfies (\ref{property1}) and (\ref{property3}), then $H=A^+$.
If $A$ has full row rank $m$ and $H$ satisfies (\ref{property1}) and (\ref{property4}), then $H=A^+$.
\end{cor}

Because of these results, we decided to focus our experiments on matrices $A$ with rank less than $\min\{m,n\}$.
We generated random dense $n\times n$  rank-$r$ matrices $A$ of the form $A=UV$, where
each $U$ and $V'$ are $n\times r$, with $n=40$, and five instance for each $r=4,8,16,\ldots,36$.
The entries in $U$ and $V$ were iid uniform $(-1,1)$. We then scaled each $A$
by a multiplicative factor of $0.01$, which had the effect of making $A+$ fully dense
to an entry-wise zero-tolerance of $0.1$. In computing various sparse pseudoinverses,
we used a zero-tolerance of $10^{-5}$. We measured sparsity of a sparse pseudoinverse
as the number of its nonzero components divided by $n^2$. We measured quality of
a sparse pseudoinverse $H$, relative to the M-P pseudoinverse $A^+$ in two ways:
\begin{itemize}
\item \emph{least-squares ratio} (`lsr'): $\| AHb-b\|_2/\| AA^+b-b\|_2$, with arbitrarily $b:=\vec{1}_{m}$.
(Note that $x:=A^+b$ always minimizes $\|Ax-b\|_2$.)
\item \emph{$2$-norm ratio} (`2nr'): $\| HA\vec{1}_{n}\|_2/\| A^+A\vec{1}_{n}\|_2$.
(Note that $x:=HA\vec{1}_{n}$ is always a solution to $Ax=A\vec{1}_{n}$, whenever $H$ satisfies
(\ref{property1}), and one that minimizes $\|x\|_2$ is given by $x:=A^+A\vec{1}_{n}$.)
\end{itemize}




\begin{prop}
If $H$ satisfies (\ref{property1}) and (\ref{property3}), then $AH=AA^+$.
\end{prop}
\begin{proof}
\begin{eqnarray*}
AHA &=& AA^+A \qquad \mbox{(by (\ref{property1}))}\\
H'A'A &=& (A^+)'A'A \qquad \mbox{(by (\ref{property3}))}\\
A'AH &=& A'AA^+\\
(A^+)'A'AH &=& (A^+)'A'AA^+\\
AH&=&AA^+,
\end{eqnarray*}
the last equation following directly from a well-known property of $A^+$.
\end{proof}
\begin{cor}\label{cor:lss}
If $H$ satisfies (\ref{property1}) and (\ref{property3}), 
 then $x:=Hb$ (and of course $A^+b$)
solves  $\min\{\|Ax-b\|_2 ~:~ x\in\mathbb{R}^n\}$.
\end{cor}
Similarly, we have the following two results:
\begin{prop}
If $H$ satisfies (\ref{property1}) and (\ref{property4}), then $HA=A^+A$.
\end{prop}
\begin{cor}\label{cor:minnorm}
If $H$ satisfies (\ref{property1}) and (\ref{property4}), and $b$ is in the column space of $A$,
then $Hb$ (and of course $A^+b$) solves $\min\{\|x\|_2 ~:~ Ax=b,~ x\in\mathbb{R}^n\}$.
\end{cor}
\noindent So in the situations covered by Corollaries \ref{cor:lss} and \ref{cor:minnorm}, we can seek and use sparser pseudoinverses than $A^+$. Our computational results are summarized in Table \ref{table1}.
`1nr' (1-norm ratio) is simply $\|H\|_1/\|A^+\|_1$. `sr' (sparsity ratio) is simply $\|H\|_0/\|A^+\|_0$.
Note that the entries of 1 reflect the results above.
We observe that sparsity can be gained versus the M-P pseudoinverse, often with a modest decrease in quality of the pseudoinverse, and we can observe some trends as the rank varies.

\begin{table}[t]
\setlength{\tabcolsep}{3.75pt}
 \renewcommand{\arraystretch}{0.8}
  \centering
  \caption{Sparsity vs quality ($m=n=40$)}
    \begin{tabular}{r|r|rrrr|rrrr|rrrr|rrrr}
    \toprule
       \multicolumn{2}{c}{}    & \multicolumn{4}{c}{P1}        & \multicolumn{4}{c}{P1+P3}       & \multicolumn{4}{c}{P1+P4}       & \multicolumn{4}{c}{P1+P3+P4} \\
    \midrule
    \multicolumn{1}{c}{r} & \multicolumn{1}{c}{$\|A^+\|_1$} & \multicolumn{1}{c}{1nr} & \multicolumn{1}{c}{sr} & \multicolumn{1}{c}{lsr} & \multicolumn{1}{c}{2nr} & \multicolumn{1}{c}{1nr} & \multicolumn{1}{c}{sr} & \multicolumn{1}{c}{lsr} & \multicolumn{1}{c}{2nr} & \multicolumn{1}{c}{1nr} & \multicolumn{1}{c}{sr} & \multicolumn{1}{c}{lsr} & \multicolumn{1}{c}{2nr} & \multicolumn{1}{c}{1nr} & \multicolumn{1}{c}{sr} & \multicolumn{1}{c}{lsr} & \multicolumn{1}{c}{2nr} \\
    \midrule
    4     & 586   & 0.44  & 0.01  & 1.07  & 2.27  & 0.60  & 0.10  & 1     & 1.43  & 0.64  & 0.10  & 1.02  & 1     & 0.75  & 0.19  & 1     & 1 \\
    4     & 465   & 0.46  & 0.01  & 1.07  & 1.82  & 0.63  & 0.10  & 1     & 1.43  & 0.63  & 0.10  & 1.01  & 1     & 0.77  & 0.19  & 1     & 1 \\
    4     & 500   & 0.44  & 0.01  & 1.08  & 1.82  & 0.62  & 0.10  & 1     & 1.46  & 0.62  & 0.10  & 1.01  & 1     & 0.76  & 0.19  & 1     & 1 \\
    4     & 503   & 0.41  & 0.01  & 1.28  & 2.00  & 0.62  & 0.10  & 1     & 1.31  & 0.62  & 0.10  & 1.06  & 1     & 0.75  & 0.19  & 1     & 1 \\
    4     & 511   & 0.45  & 0.01  & 1.10  & 2.36  & 0.63  & 0.10  & 1     & 1.55  & 0.64  & 0.10  & 1.09  & 1     & 0.78  & 0.19  & 1     & 1 \\
    8     & 855   & 0.53  & 0.04  & 1.17  & 1.63  & 0.69  & 0.20  & 1     & 1.28  & 0.68  & 0.20  & 1.05  & 1     & 0.80  & 0.36  & 1     & 1 \\
    8     & 851   & 0.53  & 0.04  & 1.22  & 1.60  & 0.69  & 0.20  & 1     & 1.33  & 0.69  & 0.20  & 1.07  & 1     & 0.80  & 0.36  & 1     & 1 \\
    8     & 841   & 0.53  & 0.04  & 1.25  & 1.70  & 0.69  & 0.20  & 1     & 1.34  & 0.69  & 0.20  & 1.07  & 1     & 0.80  & 0.36  & 1     & 1 \\
    8     & 761   & 0.52  & 0.04  & 1.05  & 1.70  & 0.69  & 0.20  & 1     & 1.32  & 0.68  & 0.20  & 1.09  & 1     & 0.81  & 0.36  & 1     & 1 \\
    8     & 864   & 0.52  & 0.04  & 1.09  & 1.40  & 0.69  & 0.20  & 1     & 1.21  & 0.68  & 0.20  & 1.04  & 1     & 0.80  & 0.36  & 1     & 1 \\
    12    & 1150  & 0.60  & 0.09  & 1.26  & 1.68  & 0.74  & 0.30  & 1     & 1.26  & 0.75  & 0.30  & 1.12  & 1     & 0.86  & 0.51  & 1     & 1 \\
    12    & 1198  & 0.59  & 0.09  & 1.20  & 1.70  & 0.75  & 0.30  & 1     & 1.25  & 0.75  & 0.30  & 1.05  & 1     & 0.85  & 0.51  & 1     & 1 \\
    12    & 1236  & 0.59  & 0.09  & 1.10  & 1.28  & 0.75  & 0.30  & 1     & 1.17  & 0.75  & 0.30  & 1.24  & 1     & 0.86  & 0.51  & 1     & 1 \\
    12    & 1134  & 0.60  & 0.09  & 1.38  & 1.43  & 0.75  & 0.30  & 1     & 1.19  & 0.74  & 0.30  & 1.09  & 1     & 0.85  & 0.51  & 1     & 1 \\
    12    & 1135  & 0.60  & 0.09  & 1.20  & 1.44  & 0.75  & 0.30  & 1     & 1.21  & 0.75  & 0.30  & 1.14  & 1     & 0.85  & 0.51  & 1     & 1 \\
    16    & 1643  & 0.67  & 0.16  & 1.36  & 1.85  & 0.79  & 0.40  & 1     & 1.30  & 0.80  & 0.40  & 1.17  & 1     & 0.90  & 0.64  & 1     & 1 \\
    16    & 1421  & 0.65  & 0.16  & 1.20  & 1.61  & 0.79  & 0.40  & 1     & 1.29  & 0.79  & 0.40  & 1.31  & 1     & 0.90  & 0.64  & 1     & 1 \\
    16    & 1518  & 0.65  & 0.16  & 1.33  & 1.38  & 0.79  & 0.40  & 1     & 1.20  & 0.80  & 0.40  & 1.30  & 1     & 0.89  & 0.64  & 1     & 1 \\
    16    & 1512  & 0.66  & 0.16  & 1.45  & 1.68  & 0.80  & 0.40  & 1     & 1.34  & 0.79  & 0.40  & 1.16  & 1     & 0.89  & 0.64  & 1     & 1 \\
    16    & 1539  & 0.65  & 0.16  & 1.18  & 1.25  & 0.79  & 0.40  & 1     & 1.19  & 0.79  & 0.40  & 1.29  & 1     & 0.89  & 0.64  & 1     & 1 \\
    20    & 2147  & 0.72  & 0.25  & 1.51  & 1.33  & 0.84  & 0.50  & 1     & 1.15  & 0.84  & 0.50  & 1.42  & 1     & 0.94  & 0.75  & 1     & 1 \\
    20    & 2111  & 0.72  & 0.25  & 1.81  & 1.44  & 0.83  & 0.50  & 1     & 1.35  & 0.84  & 0.50  & 1.48  & 1     & 0.93  & 0.75  & 1     & 1 \\
    20    & 2148  & 0.71  & 0.25  & 2.08  & 1.49  & 0.84  & 0.50  & 1     & 1.32  & 0.83  & 0.50  & 1.45  & 1     & 0.93  & 0.75  & 1     & 1 \\
    20    & 2061  & 0.72  & 0.25  & 1.50  & 1.49  & 0.84  & 0.50  & 1     & 1.35  & 0.84  & 0.50  & 1.31  & 1     & 0.93  & 0.75  & 1     & 1 \\
    20    & 2283  & 0.72  & 0.25  & 1.61  & 1.47  & 0.83  & 0.50  & 1     & 1.47  & 0.84  & 0.50  & 1.20  & 1     & 0.94  & 0.75  & 1     & 1 \\
    24    & 2865  & 0.77  & 0.36  & 1.86  & 1.24  & 0.87  & 0.60  & 1     & 1.18  & 0.87  & 0.60  & 1.51  & 1     & 0.96  & 0.84  & 1     & 1 \\
    24    & 3228  & 0.78  & 0.36  & 2.17  & 1.34  & 0.87  & 0.60  & 1     & 1.37  & 0.88  & 0.60  & 1.90  & 1     & 0.96  & 0.84  & 1     & 1 \\
    24    & 2884  & 0.77  & 0.36  & 2.27  & 1.72  & 0.87  & 0.60  & 1     & 1.32  & 0.87  & 0.60  & 1.55  & 1     & 0.96  & 0.84  & 1     & 1 \\
    24    & 2853  & 0.78  & 0.36  & 1.50  & 1.66  & 0.88  & 0.60  & 1     & 1.50  & 0.87  & 0.60  & 1.53  & 1     & 0.96  & 0.84  & 1     & 1 \\
    24    & 2944  & 0.78  & 0.36  & 1.72  & 1.48  & 0.87  & 0.60  & 1     & 1.64  & 0.88  & 0.60  & 1.64  & 1     & 0.96  & 0.84  & 1     & 1 \\
    28    & 4359  & 0.82  & 0.49  & 1.69  & 1.65  & 0.90  & 0.70  & 1     & 1.63  & 0.91  & 0.70  & 1.89  & 1     & 0.98  & 0.91  & 1     & 1 \\
    28    & 4268  & 0.83  & 0.49  & 2.27  & 1.98  & 0.91  & 0.70  & 1     & 1.79  & 0.91  & 0.70  & 2.08  & 1     & 0.98  & 0.91  & 1     & 1 \\
    28    & 4069  & 0.83  & 0.49  & 2.35  & 1.51  & 0.91  & 0.70  & 1     & 1.43  & 0.91  & 0.70  & 2.25  & 1     & 0.98  & 0.91  & 1     & 1 \\
    28    & 3993  & 0.83  & 0.49  & 2.30  & 1.58  & 0.90  & 0.70  & 1     & 1.27  & 0.91  & 0.70  & 2.19  & 1     & 0.97  & 0.91  & 1     & 1 \\
    28    & 4387  & 0.83  & 0.49  & 2.54  & 1.78  & 0.91  & 0.70  & 1     & 1.34  & 0.91  & 0.70  & 2.76  & 1     & 0.98  & 0.91  & 1     & 1 \\
    32    & 6988  & 0.88  & 0.64  & 4.08  & 1.60  & 0.94  & 0.80  & 1     & 1.81  & 0.94  & 0.80  & 3.54  & 1     & 0.99  & 0.96  & 1     & 1 \\
    32    & 6493  & 0.89  & 0.64  & 3.00  & 1.75  & 0.94  & 0.80  & 1     & 1.79  & 0.94  & 0.80  & 2.35  & 1     & 0.99  & 0.96  & 1     & 1 \\
    32    & 11445 & 0.89  & 0.64  & 4.50  & 4.82  & 0.94  & 0.80  & 1     & 2.58  & 0.94  & 0.80  & 7.18  & 1     & 0.99  & 0.96  & 1     & 1 \\
    32    & 8279  & 0.89  & 0.64  & 5.08  & 2.72  & 0.95  & 0.80  & 1     & 2.31  & 0.94  & 0.80  & 3.39  & 1     & 0.99  & 0.96  & 1     & 1 \\
    32    & 5069  & 0.89  & 0.64  & 2.14  & 1.90  & 0.95  & 0.80  & 1     & 1.74  & 0.94  & 0.80  & 2.26  & 1     & 0.99  & 0.96  & 1     & 1 \\
    36    & 18532 & 0.94  & 0.81  & 11.16 & 2.88  & 0.97  & 0.90  & 1     & 1.85  & 0.97  & 0.90  & 9.80  & 1     & 1.00  & 0.99  & 1     & 1 \\
    36    & 16646 & 0.94  & 0.81  & 10.91 & 2.53  & 0.97  & 0.90  & 1     & 3.04  & 0.97  & 0.90  & 8.07  & 1     & 1.00  & 0.99  & 1     & 1 \\
    36    & 11216 & 0.95  & 0.81  & 4.56  & 1.50  & 0.97  & 0.90  & 1     & 1.60  & 0.97  & 0.90  & 4.93  & 1     & 1.00  & 0.99  & 1     & 1 \\
    36    & 10299 & 0.95  & 0.81  & 6.12  & 1.45  & 0.98  & 0.90  & 1     & 2.14  & 0.97  & 0.90  & 5.37  & 1     & 1.00  & 0.99  & 1     & 1 \\
    36    & 11605 & 0.94  & 0.81  & 5.70  & 1.56  & 0.97  & 0.90  & 1     & 2.17  & 0.98  & 0.90  & 5.65  & 1     & 1.00  & 0.99  & 1     & 1 \\
    \bottomrule
    \end{tabular}%
  \label{table1}%
\end{table}%


\section{Conclusions and ongoing work}\label{sec:conclusion}

We have introduced eight  tractable pseudoinverses
based on using 1-norm minimization to
induce sparsity and making
convex relaxations of the M-P properties.
It remains to be seen if any of these new pseudoinverses
will be found to be valuable in practice. There is a natural tradeoff
between sparsity and closeness to the M-P properties,
and where one wants to be on this spectrum may well be
application dependent. We are in the process of carrying out
more thorough experiments.

In particular, we are
testing our sparse pseudoinverses that need semi-definite programming.
In the manner of \cite{MIQCP} and \cite{MIQCPproj}, we may go further and enforce more of (\ref{property2b}) using ``disjunctive cuts'',
producing a better convex relaxation of  (\ref{property2b}) than
(\ref{property2b_r}). But this would come at some significant costs:
(i) greater computational effort, (ii) decreased sparsity as we work our way toward enforcing more of the
M-P properties, (iii) lack of specificity.

Another idea that we are exploring is to develop update algorithms for sparse pseudoinverses.
Of course the Sherman-Morrison-Woodbury formula gives us a convenient way to update a matrix inverse of $A$ after a
low-rank modification. Extending that formula, $A^+$ can be updated efficiently (see
\cite{Meyer1973} and \cite{Riedel1992}). It is an interesting  challenge to see if
we can take advantage of a sparse pseudoinverse of $A$ in calculating a sparse pseudoinverse of a low-rank modification of $A$.

In related work, we are in the process of investigating techniques for decomposing an input matrix $\bar{C}$
into $A+B$, where $B$ has low rank and $A$ has a sparse (pseudo)inverse. In that context,
$A$ is a matrix \emph{variable}, so even the left- and right-inverse constraints
($HA = I_{n}$ and $AH = I_{m}$) are non-convex quadratic equations. So, already in that context,
we are applying similar ideas to the ones we presented here for relaxing these equations.
When we instead consider our new sparse pseudoinverses (based on the M-P
properties), again in the context where $A$ is a matrix variable, already the M-P
properties (\ref{property3}) and (\ref{property4}) are quadratic and
properties (\ref{property1}) and (\ref{property2}) are cubic (in $A$ and $H$).
We can  still apply our basic approach for handling quadratic equations, now to (\ref{property3}) and (\ref{property4}).
As for (\ref{property1}) and (\ref{property2}), we can take a variety of approaches.
One possibility is to introduce auxiliary scalar variables to get back to quadratic equations;
though even then
there are issues to consider in choosing the best way to carry this out (see \cite{SpeakmanLee2015}). Another possibility
is to introduce auxiliary \emph{matrix} variables, to get back to quadratic matrix equations
--- some fascinating questions then arise if we consider how to extend the results in \cite{SpeakmanLee2015}.
Another possibility is to employ semi-definite programming relaxations for polynomial systems (e.g., see \cite{pablo2003}).



%
%

\FloatBarrier

\bibliographystyle{plain}
\bibliography{sparsepseudo}

%
%
%
%

\end{document}